\documentclass{amsart}

\usepackage{amssymb}
\usepackage{enumitem}
\usepackage{hyperref}

\newcommand*{\mailto}[1]{\href{mailto:#1}{\nolinkurl{#1}}}
\newcommand{\arxiv}[1]{\href{http://arxiv.org/abs/#1}{arXiv:#1}}
\newcommand{\doi}[1]{\href{http://dx.doi.org/#1}{doi:#1}}
\newtheorem{theorem}{Theorem}[section]

\newtheorem{lemma}[theorem]{Lemma}
\newtheorem{corollary}[theorem]{Corollary}
\newtheorem{remark}[theorem]{Remark}
\newtheorem{hypothesis}[theorem]{Hypothesis}

\newcommand{\R}{{\mathbb R}}
\newcommand{\N}{{\mathbb N}}

\newcommand{\C}{{\mathbb C}}

\newcommand{\OO}{\mathcal{O}}

\newcommand{\be}{\begin{equation}}
\newcommand{\ee}{\end{equation}}
\newcommand{\ba}{\begin{array}}
\newcommand{\ea}{\end{array}}

\newcommand{\spr}[2]{\langle #1 , #2 \rangle}

\newcommand{\id}{{\rm 1\hspace{-0.6ex}l}}
\newcommand{\E}{\mathrm{e}}
\newcommand{\I}{\mathrm{i}}

\newcommand{\loc}{\mathrm{loc}}

\newcommand{\im}{\mathrm{Im}}
\newcommand{\re}{\mathrm{Re}}
\newcommand{\dom}{\mathfrak{D}}
\newcommand{\hil}{\mathfrak{H}}

\newcommand{\floor}[1]{\lfloor#1 \rfloor}
\newcommand{\ceil}[1]{\lceil#1 \rceil}

\newcommand{\eps}{\varepsilon}

\newcommand{\sig}{\sigma}
\newcommand{\lam}{\lambda}

\newcommand{\gam}{\gamma}

\numberwithin{equation}{section}

\begin{document}

\title[Spectral Deformations for Dirac Operators]{On Spectral Deformations and Singular Weyl Functions for One-Dimensional Dirac Operators}

\author[A.\ Beigl]{Alexander Beigl}
\address{Faculty of Mathematics\\ University of Vienna\\
Oskar-Morgenstern-Platz 1\\ 1090 Wien\\ Austria}
\email{\mailto{alex_beigl@gmx.at}}

\author[J.\ Eckhardt]{Jonathan Eckhardt}
\address{School of Computer Science \& Informatics\\ Cardiff University\\ Queen's Buildings \\ 5 The Parade\\ Roath \\ Cardiff CF24 3AA\\ Wales \\ UK}
\email{\mailto{EckhardtJ@cardiff.ac.uk}}

\author[A.\ Kostenko]{Aleksey Kostenko}
\address{Faculty of Mathematics\\ University of Vienna\\
Oskar-Morgenstern-Platz 1\\ 1090 Wien\\ Austria}
\email{\mailto{duzer80@gmail.com};\mailto{Oleksiy.Kostenko@univie.ac.at}}

\author[G.\ Teschl]{Gerald Teschl}
\address{Faculty of Mathematics\\ University of Vienna\\
Oskar-Morgenstern-Platz 1\\ 1090 Wien\\ Austria\\ and International
Erwin Schr\"odinger
Institute for Mathematical Physics\\ Boltzmanngasse 9\\ 1090 Wien\\ Austria}
\email{\mailto{Gerald.Teschl@univie.ac.at}}
\urladdr{\url{http://www.mat.univie.ac.at/~gerald/}}

\thanks{J. Math. Phys. {\bf 56}, 012102 (2015)}
\thanks{{\it Research supported by the Austrian Science Fund (FWF) under Grants No.\ J3455, P26060, and Y330}}


\begin{abstract}
We investigate the connection between singular Weyl--Titchmarsh--Kodaira theory and the double commutation method for one-dimensional Dirac operators. In particular, we compute the singular Weyl function of the commuted operator in terms of the data from the original operator. These results are then applied to radial Dirac operators in order to show that the singular Weyl function of such an operator belongs to a generalized Nevanlinna class $N_{\kappa_0}$ with $\kappa_0=\lfloor|\kappa| + \frac{1}{2}\rfloor$, where $\kappa\in \mathbb{R}$ is the corresponding angular momentum.
\end{abstract}

\maketitle

\section{Introduction}
\label{sec:int}

Weyl--Titchmarsh--Kodaira theory is the fundament of direct and inverse spectral theory for ordinary differential operators. The classical theory
usually assumes that one endpoint is regular. However, it has been shown by Kodaira \cite{ko}, Kac \cite{ka} and more recently by Fulton \cite{ful08},
Gesztesy and Zinchenko \cite{gz}, Fulton and Langer \cite{fl}, Kurasov and Luger \cite{kl}, and Kostenko, Sakhnovich, and Teschl
\cite{kst}, \cite{kst2}, \cite{kst3}, \cite{kt} that many aspects of this classical theory still can be established at a
singular endpoint. It has recently proven to be a powerful tool for inverse spectral theory for these operators
and further refinements were given in \cite{je}, \cite{je2}, \cite{egnt2}, \cite{egnt3}, \cite{CHPencil}, \cite{et}, \cite{kt2}. The analogous
theory for one-dimensional Dirac operators was developed by Brunnhuber and three of us in \cite{singdirac} (for further extensions see also \cite{ekt}, \cite{ekt2}).
Nevertheless, such operators are still difficult to understand.

One approach, originating from ideas of Krein \cite{kr57}, is to use spectral deformation methods to reduce a given spectral problem to
a simpler one. In the case of one-dimensional Schr\"odinger operators there is by now an enormous literature on this subject and
we refer to Kostenko, Sakhnovich, and Teschl \cite{kst3} for further references. Moreover, in \cite{kst3} the connection between these methods and
singular Weyl--Titchmarsh--Kodaira theory was established and it is our present aim to do the same for Dirac operators.
However, here the situation is slightly different. In fact, since Dirac operators are not bounded from below, a factorization
of the type $A^* A$ is not possible and hence there is no analog of the classical Crum--Darboux method for Dirac operators.
However, an analog of the double commutation method was established by Teschl \cite{doubledirac}. Moreover, this method
has been used by Albeverio, Hryniv, and Mykytyuk \cite{ahm2} to reduce the inverse spectral problem of radial Dirac operators on
a compact interval to the case of regular operators following the general idea of Krein \cite{kr57}. Here we will further extend these
results in a more general setting. As an application we will find a representation formula for the singular Weyl function
of radial Dirac operators and show that it is in a generalized Nevanlinna class $N_{\kappa_0}$ with $\kappa_0=\floor{|\kappa|+\frac{1}{2}}$,
extending the results from \cite{singdirac}.

\section{Singular Weyl--Titchmarsh--Kodaira Theory}
\label{sec:singth}

Let $I=(a,b) \subseteq \R$ (with $-\infty \le
a < b \le \infty$) be  an arbitrary interval. We will be concerned with Dirac operators in the Hilbert space $L^2(I,\C^2)$ equipped with the inner product
\be
\spr{f}{g}=\int_a^b f(y)^* g(y) \,dy, \quad \|f\|^2 = \spr{f}{f}.
\ee
 To this end, we consider the differential expression
\be\label{dirac}
\tau = \frac{1}{\I} \sig_2 \frac{d}{dx} + Q(x).
\ee
Here the potential matrix $Q(x)$ is given by
\be\label{diracQ}
Q(x) =  q_{\rm el}(x)\id  + q_{\rm am}(x)\sig_1 +
(m+ q_{\rm sc}(x)) \sig_3,
\ee
where $\sig_1$, $\sig_2$, $\sig_3$ denote the Pauli matrices
\be
\sig_1=\begin{pmatrix} 0 & 1 \\ 1 & 0\end{pmatrix}, \quad
\sig_2=\begin{pmatrix} 0 & -\I \\ \I & 0\end{pmatrix}, \quad
\sig_3=\begin{pmatrix} 1 & 0 \\ 0 & -1\end{pmatrix},
\ee
and $m$, $q_{\rm sc}$, $q_{\rm el}$, and $q_{\rm am}$ are interpreted as
mass, scalar potential, electrostatic potential, and anomalous
magnetic moment, respectively (see \cite[Chapter~4]{th}).
As usual, we require that $m\in[0,\infty)$ and that $q_{\rm sc}$, $q_{\rm el}$, $q_{\rm
am}
\in L^1_{\loc}(I)$ are real-valued. 

We do not include a magnetic moment $\tilde{\tau} = \tau +\sig_2 q_{\rm mg}(x)$
as it can be easily eliminated by a simple gauge transformation
$\tau = \Gamma^{-1} \tilde{\tau} \Gamma$, where $\Gamma =\exp(-\I\int^x q_{\rm mg}(r) dr)$.

If $\tau$ is in the limit point case at both $a$ and $b$, then $\tau$ gives rise to a unique
self-adjoint operator $H$ when defined maximally (cf., e.g., \cite{ls},
\cite{wdl}, \cite{wei2}). Otherwise, we fix a boundary
condition at each endpoint where $\tau$ is in the limit circle case.
Explicitly, such an operator $H$ is given by
\begin{align}\begin{split}
H: \ba[t]{lcl} \dom(H) &\to& L^2(I,\C^2) \\ f &\mapsto& \tau f \ea
\end{split}\end{align}
where
\begin{align}\begin{split} \label{domH}
\dom(H) = \{ f \in L^2(I,\C^2) \,|\, &  f \in AC_{\loc}(I,\C^2), \, \tau f \in
L^2(I,\C^2),\\ & \qquad\qquad W_a(u_-,f) = W_b(u_+,f) =0 \},
\end{split}\end{align}
with
\be
W_x(f,g) = \I \spr{f^*(x)}{\sig_2 g(x)} = f_1(x) g_2(x) - f_2(x) g_1(x)
\ee
the usual Wronskian (we remark that the limit $W_{a,b}(.,..) = \lim_{x \to
a,b} W_x(.,..)$ exists for functions as in (\ref{domH})). Here the function
$u_-$ (resp.\ $u_+$) used to generate the boundary condition at $a$
(resp.\ $b$) can be chosen to be a non-trivial solution of $\tau u =0$ if $\tau$
is in the limit circle case at $a$ (resp.\ $b$) and zero else.

For a given point $c\in I$ consider the operators $H^D_{(a,c)}$ and $H^D_{(c,b)}$
which are obtained by restricting $H$ to $(a,c)$ and $(c,b)$ with a Dirichlet boundary condition $f_1(c)=0$ at
$c$, respectively. The corresponding operators with a Neumann boundary condition $f_2(c)=0$ will be
denoted by $H^N_{(a,c)}$ and $H^N_{(c,b)}$.

Our main object will be singular Weyl--Titchmarsh--Kodaira theory which, for Dirac operators, was developed only recently in \cite{singdirac}.
Hence we will start by reviewing some relevant facts from  \cite{singdirac}. We will need a system of real entire solutions $\Phi(z,x)$, $\Theta(z,x)$ of the underlying homogeneous equation $\tau u = zu$ such that $\Phi(z,x)$ lies in the domain of $H$ near $a$ and $W(\Theta(z),\Phi(z))=1$.
To this end we introduce the following hypothesis:

\begin{hypothesis}\label{hypo}
Suppose that the spectrum of $H_{(a,c)}^D$ is purely discrete for one (and hence for all) $c\in(a,b)$.
\end{hypothesis}

\begin{lemma}[{\cite[Lemma 2.2]{singdirac}}]\label{tfae} The following are equivalent:
\begin{enumerate}[label=(\roman*), ref=(\roman*), leftmargin=*, widest=iiii] 
\item The spectrum of $H^D_{(a,c)}$ is purely discrete for some $c\in I$.
\item There is a real entire solution $\Phi(z,x)$ that is non-trivial and lies in the domain of $H$ near $a$ for each $z\in\C$.
\item There are real entire solutions $\Theta(z,x),\Phi(z,x)$ with $W(\Theta(z),\Phi(z))=1$, such that $\Phi(z,x)$ is non-trivial and lies in the domain of $H$ near $a$ for each $z\in\C$.
\end{enumerate}
\end{lemma}

Assuming Hypothesis \ref{hypo} we can introduce the singular Weyl function
\be\label{Mdef}
M(z)=-\frac{W(\Theta(z),u_+(z))}{W(\Phi(z),u_+(z))}
\ee
such that the solution which is in the domain of $H$ near $b$ is given by
\be\label{weylpsi}
\Psi(z,x)=\Theta(z,x)+M(z)\Phi(z,x).
\ee
Here, $u_+(z)$ is the (up to scalar multiples unique) non-trivial solution of $\tau u = zu$ which lies in the domain of $H$ near $b$. 
We stress the fact that there is no natural choice of a fundamental system $\Phi$ and $\Theta$. The regular solution can be multiplied by a (zero free) real entire function, $\tilde{\Phi}(z,x)=\E^{g(z)}\Phi(z,x)$ for some real entire function $g$. Due to the requirement that the Wronskian has to be normalized, the singular solution needs to be of the form $\tilde{\Theta}(z,x)=\E^{-g(z)}\Theta(z,x)-f(z)\Phi(z,x)$, where $f$ is again some real entire function. This will change the Weyl function according to
\be
\tilde{M}(z)=\E^{-2g(z)}M(z)+\E^{g(z)}f(z).
\ee
Associated with $M(z)$ is a corresponding spectral measure $\rho$ given by the Stieltjes--Liv\v{s}i\'{c} inversion formula
\be\label{defrho}
\frac{1}{2} \left( \rho\big((\lam_0,\lam_1)\big) + \rho\big([\lam_0,\lam_1]\big) \right)=
\lim_{\eps\downarrow 0} \frac{1}{\pi} \int_{\lam_0}^{\lam_1} \im\big(M(\lam+\I\eps)\big) d\lam.
\ee
Then there exists a spectral transformation which maps the Dirac operator $H$ in $L^2(I,\C^2)$ to the multiplication operator
with the independent variable in $L^2(\R,d\rho)$. Conversely, $M(z)$ can be reconstructed from $\rho$ up to an entire function.

\begin{theorem}[{\cite[Theorem 4.1]{kst2}}]\label{IntR}
Let $M(z)$ be a singular Weyl function and $\rho$ its associated spectral measure. Then there exists
an entire function $g(z)$ such that $g(\lam)\ge 0$ for $\lam\in\R$ and $\E^{-g(\lam)}\in L^2(\R, d\rho)$.

Moreover, for any entire function $\hat{g}(z)$ such that $\hat{g}(\lam)>0$ for all $\lam\in\R$ and $(1+\lam^2)^{-1} \hat{g}(\lam)^{-1}\in L^1(\R, d\rho)$
(e.g.\ $\hat{g}(z)=\E^{2g(z)}$) we have the integral representation
\be\label{Mir}
M(z) = E(z) + \hat{g}(z) \int_\R \left(\frac{1}{\lam-z} - \frac{\lam}{1+\lam^2}\right) \frac{d\rho(\lam)}{\hat{g}(\lam)},
\qquad z\in\C\backslash\sig(H),
\ee
where $E(z)$ is a real entire function.
\end{theorem}

If the endpoint $a$ is regular, or at least limit circle, then, with the usual choice for $\Phi$ and $\Theta$,
the Weyl function is a Herglotz--Nevanlinna function and we can choose
$\hat{g}(z)\equiv 1$ and $E(z)=\re(M(\I))$ in the previous theorem. However, this will not be true in general.
The following theorem gives a criterion when the singular Weyl function belongs to the class $N_\kappa^\infty$ of generalized Nevanlinna functions with no non-real poles and the only generalized pole of nonpositive type at $\infty$ (for further information on generalized Nevanlinna functions we refer to \cite{luger}, see also \cite[Appendix B]{kt}).

\begin{theorem}[{\cite[Theorem 4.3]{kst2}}]\label{thm:nkap}
Fix the solution $\Phi(z,x)$ and $k\in\N\cup\{0\}$. Then there is a corresponding solution $\Theta(z,x)$ such that $M(z)\in N_\kappa^\infty$
for some $\kappa\le k$ if and only if $(1+\lam^2)^{-k-1} \in L^1(\R,d\rho)$. Moreover, $\kappa=k$ if $k=0$ or
$(1+\lam^2)^{-k} \not\in L^1(\R,d\rho)$.
\end{theorem}

\section{The Double Commutation Method}
\label{sec:double}

As explained in the introduction, we want to investigate the effect of the double commutation method introduced for Dirac operators in \cite{doubledirac} on
the singular Weyl function. Hence we will review some prerequisites on this method first. 

Given a one-dimensional Dirac operator $H$ associated with $\tau$ in $\hil=L^2(I,\C^2)$ we denote by $u_-(z,x)$, $u_+(z,x)$ its corresponding
Weyl solutions, that is, solutions of $\tau u = z u$ which are in the domain of $H$ near $a$, $b$, respectively.
In general, $u_-(z,x)$ ($u_+(z,x)$) might not exist unless $z\in\C\backslash\sig_{ess}(H_{(a,c)}^D)$ ($z\in\C\backslash\sig_{ess}(H_{(c,b)}^D)$)
and by using them we will always implicitly suppose their existence in such a situation.
Without loss of generality we will also assume $u_\pm(z^*,x)=u_\pm(z,x)^*$ such that $u_\pm(\lam,x)$ is real whenever $\lam\in\R$.

Given $u_-(\lam,x)$, $\lam\in\R$, and $\gam\in[-\|u_-(\lam)\|^{-2},\infty]$ let us set\footnote{Here and henceforth we employ the convention $\infty^{-1}=0$.
The case $\gam=0$ has to be read as $H_0 = H$ and this case is of course trivial.}
\begin{align}
u_\gam(\lam,x)& =\frac{u_-(\lam,x)}{c_\gam(\lam,x)}, & 
c_\gam(\lam,x) & =\frac{1}{\gam}+\int_a^x u_-(\lam,y)^\top u_-(\lam,y) dy
\end{align}
and define
\begin{align}\begin{split}\label{defhgam}
H_{\gam} f = \tau_\gam f:=(\tau+Q_\gam)f, \quad \dom(H_{\gam}) = \{f \in 
\hil \,|\, f\in AC_{\loc}(I,\C^2),~ \tau_{\gam} f \in \hil, ~ \\
 \hfill W_a(u_{\gam}(\lam),f)=W_b(u_{\gam}(\lam),f)=0 \}, 
\end{split}\end{align}
\begin{align}
\begin{split}\label{Qgamma}
Q_{\gam}(x) &= \frac{2}{c_\gam(\lam,x)}\re\Big(\frac{1}{\I}\sig_2u_-(\lam,x)u_-(\lam,x)^\top\Big)\\
%\left[\frac{1}{\I}\sig_2u_-(\lam,x)u_-(\lam,x)^\top-u_-(\lam,x)u_-(\lam,x)^\top\frac{1}{\I}\sig_2\right]\\
&= \frac{u_{-,1}(\lam,x)^2
- u_{-,2}(\lam,x)^2}{c_{\gam}(\lam,x)}\sig_1 -2 \frac{u_{-,1}(\lam,x)
u_{-,2}(\lam,x)}{c_{\gam}(\lam,x)}\sig_3.
\end{split}
\end{align}
Then the main result from \cite{doubledirac} states that $H$ and $H_\gam$ are unitarily equivalent up to possibly some one-dimensional subspaces.
More precisely, denote by $P$ and $P_\gam$ the orthogonal projections onto the one-dimensional subspaces of $\hil$ spanned, respectively, by
$u_-$ and $u_\gam$ (set $P$,  $P_\gam = 0$ if $u_-$, $u_\gam\not\in \hil$). Then we have

\begin{theorem}[{\cite{doubledirac}}] \label{corspec}
Let $u_-(\lam,x)$, $\lam \in\R$, and $\gam\in[-\|u_-(\lam)\|^{-2},\infty]$ be given and define $H_\gam$ as in \eqref{defhgam}.
If $u_-(\lam) \in\hil$, then we also require $\lam\in\sig_p(H)$.

Suppose first that $u_-(\lam) \not\in \hil$.
\begin{enumerate}[label=(\roman*), ref=(\roman*), leftmargin=*, widest=iii]
\item If $\gam>0$, then $H$ and $(\id-P_{\gam}(\lam)) H_{\gam}$ are
unitarily equivalent. Moreover, $H_{\gam}$ has the additional eigenvalue $\lam$
with eigenfunction $u_{\gam,-}(\lam)$.
\item If $\gam=\infty$, then $H$ and $H_{\gam}$ are unitarily equivalent.
\end{enumerate}

Suppose that $u_-(\lam) \in \hil$ and  $\lam\in\sig_p(H)$ (i.e., $\lam$ is an eigenvalue of $H$).
\begin{enumerate}[label=(\roman*), ref=(\roman*), leftmargin=*, widest=iii]
\item If $\gam\in(-\|u_-(\lam)\|^{-2},\infty)$, then $H$ and
$H_{\gam}$ are unitarily equivalent.
\item If $\gam=-\|u_-(\lam)\|^{-2}$ or $\infty$, then $(\id-P(\lam)) H$ and
$H_{\gam}$ are unitarily equivalent, that is, the eigenvalue $\lam$ is removed.
\end{enumerate}
\end{theorem}

Furthermore, the solutions of the new operator $H_\gam$ can be expressed in terms of the solutions of $H$.

\begin{lemma}[{\cite[Lemma 3.4]{doubledirac}}]\label{lemmavpm}
Let $u \in AC_{\loc}(I,\C^2)$ fulfill $\tau u = z u$,  $z \in \C \backslash \{ \lam \}$,
and set
\be\label{vzx}
v(z,x) = u(z,x) + \frac{u_\gam(\lam,x)}{z-\lam} W_x(u_-(\lam),u(z)).
\ee
Then $v \in AC_{\loc}(I,\C^2)$ and $v$ fulfills $\tau_{\gam} v = z v$.
We also note that if $\hat{u}$, $\hat{v}$ are constructed analogously, then
\begin{align} \label{wronskis}
\begin{split}
W_x(v(z),\hat{v}(\hat{z})) &= W_x(u(z),\hat{u}(\hat{z}))
- \frac{1}{c_{\gam}(\lam,x)} \times \\ &\quad
\frac{z-\hat{z}}{(z-\lam)(\hat{z}-\lam)} 
W_x(u_-(\lam),u(z)) W_x(u_-(\lam),\hat{u}(\hat{z})).
\end{split}
\end{align}
In addition, the solutions 
\be \label{ugampm}
u_{\pm,\gam}(z,x) = u_\pm(z,x) + \frac{u_\gam(\lam,x)}{z-\lam} 
W_x(u_-(\lam),u_\pm(z)),
\ee
are square integrable  near $a,b$ and satisfy the boundary condition of $H_{\gam}$
at $a,b$, respectively.
\end{lemma}

\begin{remark}\label{remark1}
All the previous considerations still hold if one starts from the right endpoint $b$ instead of the left endpoint $a$.
One just needs to interchange the following roles:
\begin{align}
u_-(\lam,x) & \to u_+(\lam,x), &
c_\gam(x) & \to -\frac{1}{\gam}-\int_x^b u_+(\lam,y)^\top u_+(\lam,y) dy.
\end{align}
\end{remark}

Now we are ready to relate this method to singular Weyl--Titchmarsh--Kodaira theory.

\begin{theorem}\label{thm:dc1}
Let $H_\gam$ be constructed from $u_-(\lam,x)=\Phi(\lam,x)$, $\lam\in\R$, with $\gam\in[-\|\Phi(\lam)\|^{-2},\infty)$ and set $\widetilde{\Phi}_{\gam}(\lam,x)=\Phi(\lam,x)/c_\gam(\lam,x)$.
Moreover, if $\Phi(\lam)\in\hil$ we require $\lam\in\sig_p(H)$. 

The operator $H_\gam$ has a system of real entire solutions 
\begin{align} \begin{split}
\Phi_\gam(z,x) &= \Phi(z,x) -\widetilde{\Phi}_{\gam}(\lam,x) \int_a^x \Phi(\lam,y)^\top \Phi(z,y) dy\\
&= \Phi(z,x) +\frac{1}{z-\lam} \widetilde{\Phi}_{\gam}(\lam,x)W_x(\Phi(\lam), \Phi(z)), \end{split} \\
\Theta_\gam(z,x) &= \Theta(z,x) +\frac{1}{z-\lam} \Big( \widetilde{\Phi}_{\gam}(\lam,x) W_x(\Phi(\lam), \Theta(z)) + \gam \Phi_\gam(z,x) \Big),
\end{align}
with $W(\Theta_\gam(z),\Phi_\gam(z)) =1$. In fact, for $z=\lam$, we have
\begin{align}
\Phi_\gam(\lam,x) &= \gam^{-1} \widetilde{\Phi}_{\gam}(\lam,x),\label{eq:3.09}\\
\Theta_\gam(\lam,x)&=\Theta(\lam,x) +
\widetilde{\Phi}_{\gam}(\lam,x)W_x(\Phi(\lam),\dot{\Theta}(\lam)) + \gam \dot{\Phi}_\gam(\lam,x),
\end{align}
where the dot denotes the derivative with respect to the spectral parameter $z$. In particular, $H_\gam$ satisfies again Hypothesis~\ref{hypo}.

The Weyl solutions of $H_\gam$ are given by
\begin{align}\begin{split}
\Phi_\gam(z,x), \quad
\Psi_\gam(z,x) &= \Psi(z,x) +\frac{1}{z-\lam} \widetilde{\Phi}_{\gam}(\lam,x) W_x(\Phi(\lam), \Psi(z))\\
& = \Theta_\gam(z,x) + M_\gam(z) \Phi_\gam(z,x),
\end{split}\end{align}
where
\begin{align}
M_\gam(z) = M(z) - \frac{\gam}{z-\lam}
\end{align}
is the singular Weyl function of $H_\gam$.
\end{theorem}

\begin{proof}
Using Lemma~\ref{lemmavpm} it is straightforward to check that $\Phi_\gam(z,x)$, $\Theta_\gam(z,x)$ is a real entire system of solutions whose Wronskian equals
one. The extra multiple of $\Phi_\gam(z,x)$ has been added to $\Theta_\gam(z,x)$ to
remove the pole at $z=\lam$. The rest is a straightforward calculation.
\end{proof}

Note that in the previous theorem the singularity at the left endpoint is not changed, which is reflected by
the fact that also the asymptotic behavior of the Weyl function is almost unchanged.

In the limiting case $\gam=\infty$ we obtain

\begin{theorem}
Let $H_\infty$ be constructed from $u_-(\lam,x)=\Phi(\lam,x)$, $\lam\in\R$, with $\gam=\infty$ and set $\widetilde{\Phi}_\infty(\lam,x)=\Phi(\lam,x)/c_\infty(\lam,x)$.
Moreover, if $\Phi(\lam)\in\hil$ we require $\lam\in\sig_p(H)$.

The operator $H_\infty$ has a system of real entire solutions
\begin{align} 
\Phi_\infty(z,x) &= \frac{1}{z-\lam} \left(\Phi(z,x) - \widetilde{\Phi}_{\infty}(\lam,x) \int_a^x \Phi(\lam,y)^\top \Phi(z,y) dy\right),\\
\Theta_\infty(z,x) &= (z-\lam) \Theta(z,x) +\widetilde{\Phi}_{\infty}(\lam,x) W_x(\Phi(\lam), \Theta(z)),
\end{align}
with $W(\Theta_\infty(z),\Phi_\infty(z)) =1$. In fact, for $z=\lam$, we have
\begin{align}
\Phi_\infty(\lam,x) &= \dot{\Phi}(\lam,x) - \widetilde{\Phi}_{\infty}(\lam,x) \int_a^x \Phi(\lam,y)^\top \dot{\Phi}(\lam,y) dy,\\
\Theta_\infty(\lam,x)&= -\widetilde{\Phi}_{\infty}(\lam,x).
\end{align}
In particular, $H_\infty$ satisfies again Hypothesis~\ref{hypo}.

The Weyl solutions of $H_\infty$ are given by
\begin{align}\begin{split}
\Phi_\infty(z,x),\quad
\Psi_\infty(z,x) &= (z-\lam)\Psi(z,x) + \widetilde{\Phi}_{\infty}(\lam,x) W_x(\Phi(\lam), \Psi(z))\\
& = \Theta_\infty(z,x) + M_\infty(z) \Phi_\infty(z,x),
\end{split}\end{align}
where
\begin{align}
M_\infty(z) = (z-\lam)^2 M(z)
\end{align}
is the singular Weyl function of $H_\infty$.
\end{theorem}

\begin{proof}
In the limiting case $\gam \to \infty$ the definition from the previous theorem would give
$\Phi_\infty(\lam,x)=0$ (see \eqref{eq:3.09}) and we simply need to remove this zero. The rest follows as before.
\end{proof}

Note that in this case the singularity at the left endpoint is changed, however, the growth of $M(z)$ is increased whereas
it would be desirable to have a transformation which decreases the growth. Hence we need to invert
the above procedure. To this end note that the new operator $H_\infty$ has $\Theta_\infty(\lam,x)=u_{\infty,+}(\lam,x)$.
So this shows that we should look at the case where $H_\gam$ is computed from $u_+(\lam,x)=\Theta(\lam,x)$
(cf.\ Remark~\ref{remark1}). 
Let us also stress that the following result is essential for the application to the perturbed radial Dirac operator in the next section.

\begin{theorem}\label{Thetalambda}
Let $H_\gam$ be constructed from $u_+(\lam,x)=\Theta(\lam,x)\not\in\hil$, $\lam\in\R$, with $\gam\in(0,\infty]$ and set 
\begin{align}
\widetilde{\Theta}_{\gam}(\lam,x) & =\frac{\Theta(\lam,x)}{c_\gam(\lam,x)}, & 
c_\gam(\lam,x) & = -\frac{1}{\gam}-\int_x^b \Theta(\lam,y)^\top \Theta(\lam,y)dy.
\end{align}

The operator $H_\gam$ has a  system of real entire solutions
\begin{align}\label{Thetalambda:phigam}
\Phi_\gam(z,x) &= (z-\lam)\Big( \Phi(z,x) +\frac{1}{z-\lam}\widetilde{\Theta}_\gam(\lam,x) W_x(\Theta(\lam), \Phi(z))\Big),\\
\begin{split} \Theta_\gam(z,x) &= \frac{1}{z-\lam}\Big[\Theta(z,x) +\frac{1}{z-\lam}  \widetilde{\Theta}_\gam(\lam,x) W_x(\Theta(\lam), \Theta(z)) \\&\qquad\qquad + \Big(\frac{1}{\gam}-W_b(\Theta(\lam),\dot{\Theta}(\lam))\Big) \Phi_\gam(z,x) \Big],\end{split}
\end{align}
with $W(\Theta_\gam(z),\Phi_\gam(z)) =1$. Moreover,
\begin{align}\label{eq:3.23}
\Phi_\gam(\lam,x) &=\widetilde{\Theta}_\gam(\lam,x).
\end{align}
In particular, $H_\gam$ satisfies again Hypothesis~\ref{hypo}.

The Weyl solutions of $H_\gam$ are given by
\begin{align}\begin{split}
\Phi_\gam(z,x),\quad
\Psi_\gam(z,x) &= \frac{1}{z-\lam} \Big[ \Psi(z,x) +\frac{1}{z-\lam} \widetilde{\Theta}_{\gam}(\lam,x) W_x(\Theta(\lam), \Psi(z))\Big]\\
& = \Theta_\gam(z,x) + M_\gam(z) \Phi_\gam(z,x),
\end{split}\end{align}
where
\begin{align}\label{weylgamma2}
M_\gam(z)=\frac{M(z)+ W_b(\Theta(\lam),\dot{\Theta}(\lam)) (z-\lam)}{(z-\lam)^2}
-\frac{\gam^{-1}}{z-\lam}
\end{align}
is the singular Weyl function of $H_\gam$.
\end{theorem}

\begin{proof}
That $\Phi_\gam$ is entire is obvious. For $\Theta_\gam$ use l'H\^{o}pital's rule,
\begin{align*}
\lim_{z\to\lam}\Big(\Theta(z,x) +&\frac{1}{z-\lam}  \widetilde{\Theta}_\gam(\lam,x) W_x(\Theta(\lam), \Theta(z))\Big)\\&\qquad=\Theta(\lam,x)+\widetilde{\Theta}_\gam(\lam,x)W_x(\Theta(\lam), \dot{\Theta}(\lam))\\
&\qquad= -\Big(\frac{1}{\gam}-W_b(\Theta(\lam),\dot{\Theta}(\lam))\Big)\widetilde{\Theta}_\gam(\lam,x)
\end{align*}
since
\[
W_x(\Theta(\lam), \dot{\Theta}(\lam))=W_b(\Theta(\lam), \dot{\Theta}(\lam)) + \int_x^b \Theta(\lam,y)^\top\Theta(\lam,y) dy,
\]
which is obtained by differentiating the Lagrange identity
\[
(\lam-z)\int_x^b\Theta(\lam,y)^\top\Theta(z,y)\,dy=W_b(\Theta(\lam),\Theta(z))-W_x(\Theta(\lam),\Theta(z))
\]
with respect to $z$ and evaluating at $z=\lam$. Hence the pole of $\Theta_\gam(z)$ at $z=\lam$ is removed and the solution is entire. The claim about the Wronskian follows from \eqref{wronskis}. The rest follows by a straightforward calculation as before.
\end{proof}

Note that by \eqref{Mdef} we have $M(\lam)=0$ in the above situation. Moreover, the first summand in \eqref{weylgamma2} has no residue at
$z=\lam$ since the residue of $M_\gam(z)$ must be given by $-\|\Phi_\gam(\lam)\|^{-2}=-\|\widetilde{\Theta}_\gam(\lam)\|^{-2}=-\gam^{-1}$.
Furthermore, if $H$ is limit circle at $b$ and $\gam<\infty$ then $H_\gam$ will
be again limit circle at $b$ by \cite[Theorem~3.7]{doubledirac} (clearly $H_\infty$ is always limit point at $b$).
In the limit circle case, the boundary condition of $H_\gam$ will be generated by $\Phi_\gam(\lam,x) =\widetilde{\Theta}_\gam(\lam,x)\in\hil$
and hence we can repeat this procedure at every zero of
\begin{align}\begin{split}
z\mapsto W_b(\widetilde{\Theta}_\gam(\lam), \Theta_\gam(z)) = & \frac{1}{c_\gam(\lam,b)} \Big[ \frac{1}{\lam-z} W_b(\Theta(\lam),\Theta(z))\\
&{} -\Big(\frac{1}{\gam}-W_b(\Theta(\lam),\dot{\Theta}(\lam))\Big) W_b(\Theta(\lam),\Phi(z))\Big].
\end{split}\end{align}
Since we have $u_{\gam,+}(z,b)= C(z) \widetilde{\Theta}_\gam(\lam,b)$ with a nonzero entire function $C(z)$ this implies
$W(u_{\gam,+}(z), \Theta_\gam(z)) = C(z) W_b(\widetilde{\Theta}_\gam(\lam), \Theta_\gam(z))$. Now
equation \eqref{Mdef} implies that the zeros of this Wronskian coincide with the zeros of $M_\gam(z)$. But the residues of $M_\gam(z)$ are always negative and
hence there must be an odd number of zeros between two consecutive poles of $M_\gam(z)$. In particular, we see that the above Wronskian
has an infinite number of zeros and we can iterate this procedure which will be important later on. We also mention that if the function $E_\gam(z)$ in the
representation \eqref{Mir} is zero, then the derivative at every zero of $M_\gam(z)$ is positive and there will be precisely one zero between each pole.

Finally, one could also consider the case $u_+(\lam,x)=\Phi(\lam,x)$. In this case $\lam$ is an eigenvalue and the procedure coincides with the one
from Theorem~\ref{thm:dc1} if one makes the replacement $\gam^{-1} \to \gam^{-1}+\|\Phi(\lam)\|^2$.

\section{Applications to Radial Dirac Operators}
\label{sec:app}

In this section we are going to apply the double commutation method to perturbed radial Dirac operators
\be\label{begindirac}
H=\frac{1}{\I}\sig_2\frac{d}{dx}+\frac{\kappa}{x}\sig_1+Q(x),
\quad
\begin{cases}
Q(x) \in L^1_{\loc}[0,b), & |\kappa| \ne \frac{1}{2},\\
(1+|\log(x)|) Q(x) \in L^1_{\loc}[0,b), & |\kappa|=\frac{1}{2}.
\end{cases}
\ee
%where $Q\in L^1_{\loc}[0,b)$, $\kappa\in\R$ and $|\kappa|\neq \frac{1}{2}$. 
 In the case $Q\equiv 0$ the underlying differential equation can be solved in terms of Bessel functions and
in the general case standard perturbation arguments can be used to show the following:

\begin{lemma}[{\cite[Sect.~8]{singdirac}}]\label{lemPRDPhi}
If $\kappa\ge 0$, then the operator \eqref{begindirac} has a unique real entire solution satisfying
\be\label{eq:estphi}
\Phi(z,x)=x^\kappa \begin{pmatrix}0\\ \frac{\sqrt{\pi}}{2^\kappa\Gamma(\kappa+1/2)}\end{pmatrix} +o(x^\kappa)
\ee
as $x\to 0$.

In addition, this solution satisfies the growth restriction
\be\label{asymPhiprd}
|\Phi(z,x)| = \OO\bigl(|z|^{-\kappa}\E^{|\im(z)| x}\bigr)
\ee
as $|z|\to\infty$ for all $x$ and has the asymptotics 
\be\label{asymPhiprd2}
\Phi(z,x) \sim  
(\pm z)^{-\kappa}\begin{pmatrix}\sin\bigl(z x \mp \frac{\kappa \pi}{2} - \int_0^x q_{\rm el}(y) dy\bigr)\\ \cos\bigl(z x \mp \frac{\kappa \pi}{2}- \int_0^x q_{\rm el}(y) dy\bigr)\end{pmatrix}
\ee
as $|z|\to\infty$ in any sector $|\arg(\pm z)| < \pi - \delta$.
\end{lemma}

\begin{lemma}[{\cite[App.~A]{ahm2}}]\label{lem:sol}
If $\kappa>\frac{1}{2}$, then \eqref{begindirac} has a second real entire solution $\Theta(z,x)$ with $W(\Theta(z),\Phi(z))=1$
satisfying
\be
\Theta(z,x)=\begin{pmatrix} x^{-\kappa}\widetilde{\Theta}_1(z,x)\\x^{-\kappa+1}\widetilde{\Theta}_2(z,x)\end{pmatrix},
\ee
with $\widetilde{\Theta}_1(z,.)\in C[0,b)$ and $\widetilde{\Theta}_2(z,.)\in L_{\loc}^1[0,b)$.

Moreover, 
\be
\int_x^c|\Theta(z,y)|^2\,dy=x^{-2\kappa+1}\widetilde{w}(z,x),\quad \widetilde{w}(z,.)\in W^{1,1}(0,c),\quad \widetilde{w}(z,0)>0,
\ee
for every $c\in(0,b)$.
\end{lemma}

Now our strategy is the usual one (cf.\ also \cite{ahm2}): We iteratively apply the double commutation method to lower $\kappa$ until we
end up in the limit circle case $|\kappa|\in[0,1/2)$. To be able to satisfy the requirement $u_+(\lam,x)=\Theta(\lam,x)$ from
Theorem~\ref{Thetalambda} we will assume that the right endpoint $b$ is regular.

\begin{lemma}\label{lemmaHgamma}
Let $H$ be given by \eqref{begindirac} with $\kappa>\frac{1}{2}$. Moreover,
let $H_\gam$ be constructed from $u_+(\lam,x)=\Theta(\lam,x)$, $\lam\in\R$, with $\gam\in(0,\infty]$ as in Theorem~\ref{Thetalambda}.
Then
\be
H_\gam=\frac{1}{\I}\sig_2\frac{d}{dx}+\frac{1-\kappa}{x}\sig_1+\widetilde{Q}(x),\quad \widetilde{Q}\in L^1_{\loc}[0,b).
\ee
Moreover, if $\kappa\ge 1$ and $\Phi(z,x)$ is normalized according to Lemma~\ref{lemPRDPhi}, then so is $\I\sigma_2\Phi_\gam(z,x)$.
\end{lemma}

\begin{proof}
By \eqref{Qgamma} and Remark~\ref{remark1}, the commuted operator is of the form $H_\gam=H+Q_\gam$, where
\begin{align*}
Q_\gam(x)&=\frac{\Theta_1(\lam,x)^2
- \Theta_2(\lam,x)^2}{c_{\gam}(\lam,x)}\sig_1 -2 \frac{\Theta_1(\lam,x)
\Theta_2(\lam,x)}{c_{\gam}(\lam,x)}\sig_3\\
&=\frac{c_{\gam}'(\lam,x)}{c_{\gam}(\lam,x)}\sig_1-\frac{2\Theta_2(\lam,x)^2}{c_{\gam}(\lam,x)}\sig_1-2 \frac{\Theta_1(\lam,x)
\Theta_2(\lam,x)}{c_{\gam}(\lam,x)}\sig_3.
\end{align*}
By Lemma~\ref{lem:sol}, the denominator is of the form
\[
c_{\gam}(\lam,x)=-\frac{1}{\gam}-\int_x^b|\Theta(\lam,y)|^2\,dy=x^{-2\kappa+1}\Big(-\frac{1}{\gam}x^{2\kappa-1}+w(x)\Big),
\]
with $w\in W^{1,1}(0,b)$ and $w>0$ on $[0,b]$. Note that since $\kappa>\frac{1}{2}$, the mapping $x\mapsto x^{2\kappa-1}$ lies in $W^{1,1}(0,b)$ too and therefore $c_{\gam}(\lam,x)=x^{-2\kappa+1}\widetilde{w}(x)$, where $\widetilde{w}$ shares the same properties as $w$.
Hence
\[
\frac{c_{\gam}'(\lam,x)}{c_{\gam}(\lam,x)}=\frac{d}{dx}\log(c_{\gam}(\lam,x))=\frac{-2\kappa+1}{x}+\frac{\widetilde{w}'(x)}{\widetilde{w}(x)},
\]
with $\widetilde{w}'/\widetilde{w}\in L^1(0,b)$. Using the properties of the singular solution from Lemma~\ref{lem:sol},
one infers that $\Theta_2(\lam,.)^2/c_{\gam}(\lam,.)$ and $(\Theta_1(\lam,.)
\Theta_2(\lam,.))/c_{\gam}(\lam,.)$ lie in $L^1(0,b)$ too and the first part follows.

To see the last part, recall that every entire solution of $H_\gam$ which lies in the domain of $H_\gam$ near $a=0$ is of the form
$\E^{g(z)} \Phi_\gam(z,x)$. Since $\Phi_\gam(z,x) \sim  z \Phi(z,x)$ as $|z|\to \infty$, equations \eqref{asymPhiprd} and \eqref{asymPhiprd2} show that $g(z)\equiv 0$.
\end{proof}

\begin{remark}
The new operator $H_\gam$ has a negative angular momentum if $\kappa>1$. In this case we employ the gauge transform $\sig_2 H_\gam \sig_2$, resulting in a positive angular momentum. The corresponding system of fundamental solutions is given by $\I\sig_2\Phi_\gam(z,x)$ and $\I\sig_2\Theta_\gam(z,x)$. Note that we again have $W(\I\sig_2\Theta_\gam(z),\I\sig_2\Phi_\gam(z))=W(\Theta_\gam(z),\Phi_\gam(z))=1$ and the formula $\eqref{weylgamma2}$
remains unchanged.
If $\kappa\in(1/2,1]$, then $1-\kappa\in[0,1/2)$ but $\Phi_\gam(z,x)$ will not be normalized according to Lemma~\ref{lemPRDPhi} but will correspond to a different boundary condition
(e.g., for $\kappa=1$ it corresponds to the boundary condition $\Phi_\gam(z,0)=(1,0)$).
\end{remark}

In order to iterate this procedure we will assume that our operator is regular at $b$. Then $H_\gam$ will again be regular at $b$ as long as
$\gam\in(0,\infty)$. Moreover, by the discussion after Theorem~\ref{Thetalambda} there will be another choice $\hat\lam$ such that
$u_{\gam,+}(\hat{\lam})=\Theta_\gam(\hat{\lam})$ (i.e., such that $\Theta_\gam(\hat{\lam})$ satisfies the boundary condition of $H_\gam$ at $b$).

Since the singular Weyl function in the limit circle case will be a Herglotz function, combining these results with Theorem~\ref{Thetalambda}, one obtains by induction:
\begin{theorem}\label{final1}
Let $H$ be given by \eqref{begindirac} with $\kappa>\frac{1}{2}$, $\kappa+\frac{1}{2}\notin\N$, and let $b$ be regular. Assume also that $\Phi(z,x)$ is normalized according to Lemma~\ref{lemPRDPhi}. 
Then there is a singular Weyl function of the form
\be\label{Msum}
M(z)=P_{\floor{\kappa+\frac{1}{2}}}(z)^2 M_0(z)-\sum_{n=0}^{\floor{\kappa-\frac{1}{2}}} c_n P_n(z)^2 (\lam_n-z),
\ee
where $M_0(z)$ is a Herglotz--Nevanlinna function and
\begin{align}
c_n &=\gam_n^{-1}-W_b(\Theta_n(\lam_n),\dot{\Theta}_n(\lam_n)),\\
P_n(z) &=\prod_{j=0}^{n-1} (z-\lam_j),\quad P_0(z)=1,
\end{align}
depends on the choice of $\lam_n$ and $\gam_n$ in every step of Lemma~\ref{lemmaHgamma}.
The corresponding spectral measure is given by
\be\label{rhosum}
d\rho(t)=P_{\floor{\kappa+\frac{1}{2}}}(t)^2 d\rho_0(t),
\ee
where the measure $\rho_0$ satisfies $\int_\R d\rho_0(t)=\infty$ and $\int_\R \frac{d\rho_0(t)}{1+t^2}<\infty$.
\end{theorem}

\begin{proof}
As pointed out before we can reduce $\kappa$ by $1$ using the above method until we reach the case $\kappa\in(1/2,3/2)$.
If $\kappa\in(1,3/2)$, the above procedure will lead us to $\kappa\in[0,1/2)$ with a properly normalized $\Phi$ and the theorem
is proven. In the case $\kappa\in(1/2,1)$ Lemma~\ref{lemmaHgamma} will give us an operator $H_\gamma$ of type \eqref{begindirac}
with $\kappa\to 1-\kappa$. Moreover, as in the proof of the previous lemma we see
\[
\widetilde{\Theta}_\gam(\lam,x)=x^{\kappa-1} \begin{pmatrix}C\\ 0\end{pmatrix} +o(x^{\kappa-1}), \qquad C\ne 0,
\]
and by inspection of \eqref{Thetalambda:phigam} (see also \eqref{eq:3.23}) we see
\[
\Phi_\gam(z,x)=x^{\kappa-1} \begin{pmatrix}C\\ 0\end{pmatrix} +o(x^{\kappa-1}).
\]
Moreover, from $W(\Theta_\gam(z),\Phi_\gam(z))=1$ we conclude
\[
\Theta_\gam(z,x)=x^{1-\kappa} \begin{pmatrix}0\\ C^{-1}\end{pmatrix} +o(x^{1-\kappa}).
\]
The last two equations imply
\[
W_0(\Theta_\gam(z),\Theta_\gam(\hat{z})) =0, \qquad W_0(\Theta_\gam(z),\Phi_\gam(\hat{z})) =1,
\]
whereas \eqref{wronskis} implies
\[
W_0(\Phi_\gam(z),\Phi_\gam(\hat{z})) =0
\]
Hence a direct computation (cf.\ the proof of Theorem A.7 in \cite{kst2}) shows
\[
\im(M_\gam(z)) = \im(z) \int_0^b |\psi_\gam(z,x)|^2 dx,
\]
which implies that the corresponding Weyl function $M_\gam(z)$ is a Herglotz--Nevanlinna function as required.
\end{proof}

As another consequence we obtain

\begin{corollary}\label{final2}
Let $H$ be given by \eqref{begindirac} with $\kappa>\frac{1}{2}$, $\kappa+\frac{1}{2}\notin\N$, and let $b$ be regular.
Then there is a corresponding system of entire solutions $\Phi(z,x)$, $\Theta(z,x)$ with $\Phi$ as in Lemma~\ref{lemPRDPhi}
such that $M(z)\in N_{\kappa_0}^\infty$ with $\kappa_0=\floor{\kappa+\frac{1}{2}}$.
\end{corollary}

\begin{proof}
Combining \eqref{rhosum} with $\int (1+t^2)^{-1}d\rho_0(t)<\infty$ and $\int d\rho_0(t)=\infty$, the claim follows by applying Theorem~\ref{thm:nkap} with $\kappa_0=\floor{\kappa+\frac{1}{2}}$.
\end{proof}

\begin{remark}
It is easy to see that the assumption that $b$ is regular is superfluous. Indeed, observe that Lemma~7.1 from \cite{singdirac} shows that
the asymptotics of $M(z)$ as $\im(z)\to\infty$ depend only on the behavior of the potential near $a=0$. Furthermore, \cite[Lemma C.2]{kst2}
shows that the required integrability properties of the spectral measure $d\rho$ depend only on the asymptotics of $M(z)$ and hence
also depend only on the behavior of the potential near $a=0$.
\end{remark}

This generalizes Theorem~8.4 from \cite{singdirac} where the bound $\kappa_0\le\ceil{\kappa}$ was given.

\begin{remark}
There is a straightforward connection with the standard theory for radial Schr\"odinger operators
 if our Dirac operator is supersymmetric, that is, $q_{\rm el}= q_{\rm sc}=0$ (see \cite[Section 3]{singdirac}). Using the results of this section, we can extend Theorem 4.5 from \cite{kt} to Schr\"odinger operators defined in $L^2(0,b)$ by differential expressions 
\be\label{eq:4.10}
\ell=a_q a_q^*:= \left( - \frac{d}{dx} + \frac{\kappa}{x} + q_{\rm am}(x)\right)\left( \frac{d}{dx} + \frac{\kappa}{x} + q_{\rm am}(x)\right).
\ee
Note that in the case $q_{\rm am }\in L^2_{\loc}[0,b)$ this differential expression can be written in the potential form 
\be
\ell=-\frac{d^2}{dx^2}+\frac{\kappa(\kappa+1)}{x^2}+q_+,\quad q_+(x)=\frac{2\kappa}{x}q_{am}(x)+q_{am}(x)^2-q_{am}'(x),
\ee
where $q_+$ is a $W^{-1,2}_{\loc}[0,b)$ distribution. For further details we refer to \cite{ekt2}.
\end{remark}

\bigskip
\noindent
{\bf Acknowledgments.}
We are very grateful to Annemarie Luger for helpful discussions.
A.B., J.E., and G.T.\ gratefully acknowledge the stimulating atmosphere at the {\em Institut Mittag-Leffler} during summer 2014 where parts of this paper were written during the workshop on {\em Modern aspects of the Titchmarsh--Weyl m-function and its multidimensional analogues}.

\end{document}